\documentclass[12pt]{article}

\usepackage[utf8]{inputenc}

\usepackage[pagebackref=true, colorlinks, linkcolor=blue, citecolor=red]{hyperref}
\usepackage{amsmath,amsthm,amsfonts,amssymb,amscd,authblk,mathrsfs,enumerate,doc,nameref}
\usepackage{fancyhdr, tikz,caption,makeidx} 
\usepackage{etoolbox}

\makeatletter
\newcounter{author}
\renewcommand*\author[1]{%
	\stepcounter{author}%
	\ifnum\c@author=1
	\gdef\@author{#1}%
	\else
	\xdef\@author{\unexpanded\expandafter{\@author \  and \ #1}}%
	\fi
	\csgdef{author@\the\c@author}{#1}}
\newcommand*\email[1]{%
	\csgdef{email@\the\c@author}{#1}}
\newcommand*\address[1]{%
	\csgdef{address@\the\c@author}{#1}}
\AtEndDocument{%
	\xdef\author@count{\the\c@author}%
	\c@author=1
	\print@authors}
\newcommand*\print@authors{%
	\ifnum\c@author>\author@count
	\else
	\print@author{\the\c@author}%
	\advance\c@author by 1
	\expandafter\print@authors
	\fi}
\newcommand*\print@author[1]{%
	\par\medskip
	\begin{tabular}{@{}l@{}}%
		\csuse{address@#1}\\
		\textit{E-mail address}:
		\href{mailto:\csuse{email@#1}}{\csuse{email@#1}}
\end{tabular}}
\makeatother

\makeindex

\theoremstyle{plain}
\newtheorem{theorem}{Theorem}[section]

\newtheorem{lemma}[theorem]{Lemma}
\newtheorem{proposition}[theorem]{Proposition}

\newtheorem{corollary}[theorem]{Corollary}
\newtheoremstyle{named}{}{}{\itshape}{}{\bfseries}{.}{.5em}{\thmnote{#3's }#1}
\theoremstyle{named}
\newtheorem{namedtheorem}{Theorem}
\theoremstyle{definition}
\newtheorem{remark}[theorem]{Remark}

\newtheorem{example}[theorem]{Example}
\newtheorem{definition}[theorem]{Definition}
\numberwithin{equation}{section}
\newcommand{\bt}{\begin{theorem}}
	\newcommand{\et}{\end{theorem}}
\newcommand{\bco}{\begin{corollary}}
	\newcommand{\eco}{\end{corollary}}
\newcommand{\bd}{\begin{definition}}
	\newcommand{\ed}{\end{definition}}
\newcommand{\bp}{\begin{problem}}
	\newcommand{\ep}{\end{problem}}
\newcommand{\bl}{\begin{lemma}}
	\newcommand{\el}{\end{lemma}}
\newcommand{\bprop}{\begin{proposition}}
	\newcommand{\eprop}{\end{proposition}}
\newcommand{\br}{\begin{remark}}
	\newcommand{\er}{\end{remark}}
\newcommand{\bpf}{\begin{proof}}
	\newcommand{\epf}{\end{proof}}
\newcommand{\bex}{\begin{example}}
	\newcommand{\eex}{\end{example}}

\newcommand{\del}{\mathcal{D}\mathcal{E}\mathcal{L}}

\newcommand{\di}{\mathcal{D}}

\newcommand{\C}{\mathbb{C}}

\newcommand{\oF}{\overline{F}}
\renewcommand{\t}{\theta}
\renewcommand{\a}{\alpha}
\renewcommand{\b}{\beta}

\title{Integration in Finite Terms: Dilogarithmic Integrals}

\author{Yashpreet Kaur %\thanks{\noindent The results of this paper are part of author's doctoral dissertation.}
}
\address{Mathematics Department, Indian Institute of Science Education and Research Pune}
\email{yashpreetkm@gmail.com}

\author{Varadharaj R. Srinivasan}
\address{Department of Mathematics, Indian Institute of Science Education and Research Mohali}
\email{ravisri@iisermohali.ac.in}

\begin{document}
\date{}
\maketitle
%%%%%%%%%%%%%%%%%%%%%%%%%%%%%%%%%%%%%%
%%%%%%%%%%%%%%%%%%%%%%%%%%%%%%%%%%%%%%
\begin{abstract}
We extend the theorem of Liouville on integration in finite terms to include dilogarithmic integrals. The results provide a necessary and sufficient condition for an element of the base field to have an antiderivative in
a field extension generated by transcendental elementary functions and  dilogarithmic integrals. We also study algebraic independence of certain dilogarithmic integrals.
\end{abstract}
%%%%%%%%%%%%%%%%%%%%%%%%%%%%%%%%%%%%%%
%%%%%%%%%%%%%%%%%%%%%%%%%%%%%%%%%%%%%%
\section{Introduction}\label{intro}
 
In this paper, a field always means a field of characteristic zero and differential fields are equipped with a single derivation map denoted by $'$.   If $F$ is a differential field then the kernel $C_F$ of the derivation map is a subfield of $F$, called the field of constants of $F$. 
Let $\t, \eta\in F$ with $\eta\neq 0$ and $\t'=\eta'/\eta$.  Then $\theta$ is called a \emph{logarithm} of $\eta$ and $\eta$ is called an \emph{exponential} of $\theta.$ Note that for any $c\in C_F,$ we have $(\theta+c)'=\eta'/\eta$ and for any nonzero $c\in C_F,$ we have $\theta'=(c\eta)'/(c\eta).$ For convenience and clarity, if an element $\theta$ is a logarithm (respectively, an exponential) of $\eta$ then we shall use the symbol $\log(\eta)$ (respectively, $e^\eta$) to denote $\theta$.  A differential field extension $E=F(\t_1,\dots, \t_n)$ is called an \emph{elementary extension} of $F$ if  either $\t_i$ is a logarithm of an element of $F_{i-1}:=F(\t_1,\dots, \t_{i-1})$ or $\t_i$ is an exponential of an element of $F_{i-1}$ or $\t_i$ is algebraic over $F_{i-1},$ where $F_0:=F.$ An element $f\in F$ is said to admit an \emph{elementary integral} if there is an elementary extension $E$ of $F$ with $C_E=C_F$ such that $u'=f$ for some $u\in E;$ in which case, $u$ is called an \emph{elementary integral} of $f.$  

In \cite{Ros1968},  Rosenlicht provided a purely algebraic  necessary and sufficient criterion for a function to admit an elementary integral. This criterion, often referred in the literature as Liouville's Theorem on integration in finite terms,  states that if  $f\in F$ has an elementary integral then there are constants $c_1,\dots, c_n\in C_F$ and elements $w, r_1,\dots, r_n\in F$ such that \begin{equation}\label{liouvilleeqn} f=\sum^n_{i=1}c_i\frac{r'_i}{r_i}+w'.
\end{equation} Thus, if $u$ is an elementary integral of $f$ then there is an element $w\in F$ such that $u-w$ is a $C_F-$linear combination of logarithms of elements of $F.$ Ever since the publication of Liouville's Theorem,  several mathematicians  extended the theorem to include special functions such as logarithmic integrals, dilogarithmic integrals, Fresnel integrals and error functions.  A detailed account of the history of this problem can be found in \cite{Badd1}, \cite{Ritt} and \cite{sin-saund-cav}.  In this paper we consider the problem of extending  Liouville's Theorem by allowing dilogarithmic integrals in addition to transcendental elementary functions. 

{\bf }Let $g\in F\setminus\{0, 1\}$. An element $\ell_{2}(g)\in F$ is called a \emph{dilogarithm} (See \cite{Badd1}, p.912) of $g$ or a \emph{dilogarithmic integral} of $g$ if there is an  element $\theta \in F$ such that $\theta'=-(1-g)'/(1-g)$ and  $\ell'_{2}(g)=(g'/g)\theta$.  A differential field extension $E=F(\theta_1,\dots,\theta_n)$ of $F$ is a {\it dilogarithmic-elementary extension} of $F$  if  $\theta_i$ is algebraic over $F_{i-1}$ or $\t_i$ is exponential of an element of $F_{i-1}$ or $\t_i$ is a logarithm of an element of $F_{i-1}$ or there is an element $g\in F_{i-1}\setminus\{0,1\}$ such that $\theta_i$ is a dilogarithm of $g$. If none of the $\theta_i$ are algebraic over $F_{i-1}$ then we say $E$ is a {\it transcendental dilogarithmic-elementary extension} of $F.$ The following  theorem, which deals with the problem of integration in finite terms with dilogarithmic integrals, is the main result of this paper and its proof is contained in Theorem \ref{maintheorem}.

\bt \label{NSdilog}
Let  $E=F(\theta_1,\dots,\theta_n)$ be a transcendental dilogarithmic-elementary extension of $F$, $C_E=C_F$ and $u\in E$. Then $u'\in F$ if and only if  there is a finite indexing set $I$ and elements $r_i, w\in F$ and $g_i\in F\setminus\{0,1\}$ such that
\begin{align}
u'&=\sum_{i\in I}r_i\frac{g'_i}{g_i}+ w',\label{NSdilogeqn}\\
\notag r'_i&=c_i\frac{(1-g_i)'}{1-g_i}+\sum_{j\in I}c_{ij}\frac{g'_j}{g_j},
\end{align}
where $c_i$ and $c_{ij}$ are constants such that $c_{ij}=c_{ji}$. 

\et

The problem of integration in finite terms with dilogarithmic integrals was first considered by Baddoura (See \cite{Badd1}, p.933),  where he proved the following theorem: If $E$ is a transcendental dilogarithmic-elementary extension of $F,$  $C_E=C_F$, $C_F$ is an algebraically closed field, $F$ is a liouvillian extension of $C_F$ and $u\in E$ is an element such that $f:=u'\in F$ then \begin{equation}\label{expforu}u=\sum_{i\in I}c_iD(g_i)+ \sum_{j\in J}s_j\log(v_j)+w, \end{equation}where $g_i, s_j, v_j,w$ are elements in $F$, $c_i\in C_F$ and $D(g_i)=\ell_2(g_i)+(1/2)\log(g_i)\log(1-g_i).$
 In \cite{ykvrs-2019}, p.228 , Theorem 6.1, we devised new techniques that allowed us to generalize\footnote{We no longer require that $C_F$ is an algebraically closed field or that $F$ is a liouvillian extension of $C_F$.} and provide a simpler proof of this theorem of Baddoura. In a recent article (See \cite{Heb}), Hebisch has  generalized and reproved results of Baddoura on dilogarithmic integrals using completely different techniques that involve certain tensor product constructions. The generalization is that he also allows algebraic elements in the field of definition and thus the hypothesis that $\theta_i$ must be transcendental over $F_{i-1}$ in Baddoura's theorem shall be dropped. Note that Equation \ref{expforu} only provides an expression for $u$ in terms of dilogarithms and logarithms of elements of $F.$ Our result, in the spirit of Liouville's Theorem,   provides a necessary and sufficient condition (See Equations \ref{NSdilogeqn}) for the element $u'\in F$ to admit an antiderivative in a transcendental dilogarithmic-elementary extension. 
 
In Section \ref{sectionidndilog}, we provide  a new dilogarithmic identity (See Proposition \ref{dilogarithmicidentity} \ref{newdilogidentity}) and show that for distinct elements $\a_1,\dots,\a_t \in F$ and $~t\geq3,$ the set $\{\ell_2\left(\frac{\t-\a_j}{\t-\a_k}\right);k>j\}$ is algebraically independent over the logarithmic extension  $E=F(\t)(\{\log(\a_j-\a_k),\log(\t-\a_j);j,k=1,\dots,t\}).$  
The necessary differential algebra preliminaries required to read this article can be found in Section 2 of  \cite{ykvrs-2019}.  

%%%%%%%%%%%%%%%%%%%%%%%%%%%%%%%%%%%%%%%%%%%%%%%%%%%%%%%%%%%%%%%%
\section{Integration in $\del$ extensions}\label{delextns}

%%%%%%%%%%%%%%%%%%%%%%%%%%%%%%%%%%%%%%%%%%%%
An extension $E$ is called a  {\it $\del-$extension} (\cite{ykvrs-2019}, pp. 210-211) of $F$ if $E=F(\theta_1,\dots,\theta_n)$, $F_0:=F$ and $F_i=F_{i-1}(\t_i)$ with $C_E=C_F$, where
$\theta_i$ is algebraic over $F_{i-1}$ or $\theta_i$ is an exponential of an element of $F_{i-1}$ or $\theta_i$ is a logarithm of an element of $F_{i-1}$ or $\theta_i$ is a dilogarithmic integral of an element of $F_{i-1}$ or $\theta_i$ is an error function of an element of $F_{i-1}$ (that is $\theta'_i=u'v$, where $v'=(-u^2)'v$ for some $u,v\in F_{i-1}$) or $\theta_i$ is a logarithmic integral of an element of $F_{i-1}$ (that is, $\theta'_i=u'/v$, where $v'=u'/u$ for some $u,v\in F_{i-1}$).

The main result of this section is Theorem \ref{maintheorem}, where we shall classify all elements of the field $F$ that admits an antiderivative in a transcendental $\del-$extension of $F$. As a corollary, we shall obtain a proof of Theorem \ref{NSdilog}. We first recall from \cite{ykvrs-2019} the following definitions and a theorem.

An element $v\in F$ is said to admit a {\it $\del-$expression}  over $F$ if there are finite indexing sets $I,J,K$ and elements    $r_i\in F,$ $g_i\in F\setminus\{0,1\}$ for all $i\in I$, elements $u_j,\log(u_j)\in F$ and constants $a_j$ for all $j\in J$, elements $v_k, e^{-v^2_k}\in F$ and constants $b_k$ for all $k\in K$, and an element $w\in F$   such that  \begin{equation}\label{deflde}v=\sum_{i\in I} r_i\frac{g'_i}{g_i}+\sum_{j\in J}a_j\frac{u'_j}{\log(u_j)}+\sum_{k\in K}b_kv'_ke^{-v^2_k}+w',\end{equation} where for  each $i\in I$, there is an integer $n_i$ such that  $r'_i= \sum^{n_i}_{l=1}c_{il} h'_{il}/h_{il}$  for some constants $c_{il}$ and elements $h_{il}\in F$.  A $\del-$expression is called  a \textit{special $\del-$expression}  if for each $i\in I$, $r'_i=c_i(1-g_i)'/(1-g_i)$ for some constant $c_i$ and a {\it $\mathcal{D}-$expression} if it is special and for each $j,k,$ $a_j=b_k=0.$

%%%%%%%%%%%%%%%%%a%%%%%%%%%%%%%%%%%%%%%%%%%%%%%%%%%%%%
%%%%%%%%%%%%%%%%a%%%%%%%%%%%%%%%%%%%%%%%%%%%%%%%%%%%%

%%%%%%%%%%%%%%%%%%%%%%%%%%%%%%%%%%%%%
A differential field extension $E$ of $F$ is called a {\it logarithmic extension} of $F$ if $C_E=C_F$ and there are elements $h_1,\dots,h_m\in F$ such that $E=F(\log h_1,\dots,\log h_m)$.

%%%%%%%%%%%%%%%%%%%%%%%%%%%%%%%%%%%%%%

%%%%%%%%%%%%%%%%%%%%%%%%%%%%%%%%%%%%%%

%%%%%%%%%%%%%%%%%%%%%%%%%%%%%%%%%%%%%%%%%%
%%%%%%%%%%%%%%%%%%%%%%%%%%%%%%%%%%%%%%%%%%

\bt\label{submaindilog}{\bf (\cite{ykvrs-2019}, p.227)} Let  $E=F(\theta_1,\dots,\theta_n)$ be a transcendental $\del-$extension of $F.$ Suppose that there is an element $u\in E$ with $u'\in F$ then $u'$ admits a special $\del-$expression over some logarithmic extension of $F$ and a $\del-$expression over $F$. Furthermore, if $E$ is a transcendental dilogarithmic-elementary extension of $F$ then $u'$ admits a $\di-$expression over some logarithmic extension of $F$. 
\et

Observe that Theorem \ref{submaindilog} provides a necessary condition for an element to admit an antiderivative in a transcendental $\del-$extension.  In the next two Propositions,  we modify Theorem \ref{submaindilog} so as to obtain a criteria that is both necessary as well as sufficient.  One of the important ingredients in the next proposition is a version of the Kolchin-Ostrowski Theorem,  whose statement we shall reproduce here for the benefit of the reader.
\begin{namedtheorem}[Kolchin-Ostrowski] (\cite{Kol1968}, p.$1155$  or \cite{Rue-Sin}, Appendix.)\label{kol-ost}
Let $E=F(\theta_1,\theta_2,\dots,\theta_n)$ be a  differential field extension of $F$ with $\theta'_i\in F$ for each $i$  and  $C_E=C_F$. If $y\in E$ and  $y'\in F$ then $y=\sum^n_{i=1}c_i\theta_i+\eta$, for constants $c_1,\cdots, c_n$ and $\eta\in F$. \end{namedtheorem}

%%%%%%%%%%%%%%%%%%%%%%%%%%%%%%%%%%%%%%%%%%%%%%%%%5555%5%55555555%%%%%%%%%

\begin{theorem} \label{specialtodel}
	Let $F$ be a differential field and $v\in F$. If $v$ admits a special $\del-$expression over a logarithmic extension $E$ of $F$ then $v$ admits a $\del-$expression over $F$ of the following form.
	\begin{align}\label{dilogv}
	v=\sum_{i\in I}r_i\frac{g'_i}{g_i}+ \sum_{j\in J}a_j\frac{u'_j}{\log u_j}+\sum_{k\in K}b_kv'_ke^{-v^2_k}+w',
	\end{align}
	\begin{equation}\label{dilogcoef}
	r'_i=c_i\frac{(1-g_i)'}{1-g_i}+\sum_{l\in I}c_{il}\frac{g'_l}{g_l},
	\end{equation}
	where $c_i$ and $c_{il}$ are constants such that $c_{il}=c_{li}$.
\end{theorem}

\begin{proof}
	Assume that $v$ admits special $\del-$expression over some logarithmic extension of $F$ and proceed as in the proof of Lemma 4.3 of \cite{ykvrs-2019}, with $F$ in place of $F(\theta)$, to obtain  the following special $\del-$expression for $v.$
	\begin{align}\label{eqndelv}
	v=\sum_{i\in I}r_i\frac{g'_i}{g_i}+\sum_{j\in J}a_j\frac{u'_j}{\log u_j}+\sum_{k\in K}b_kv'_ke^{-v_k^2}+w',
	\end{align}
	where  each $g_i, u_j, \log u_j, v_k, e^{-v^2_k}$ belongs to $F,$ each $r_i$ and $w$ belong to some logarithmic extension $E$ of $F$ and for some constant $c_i$, \begin{equation}\label{newr_i}
		r'_i=c_i\frac{(1-g_i)'}{1-g_i}.\end{equation}
	
Let $S:=v-\sum_{j\in J}a_j\frac{u'_j}{\log u_j}+\sum_{k\in K}b_kv'_ke^{-v_k^2}\in F$ and observe that
		\begin{align}\label{eqndel}
	S=\sum_{i\in I}r_i\frac{g'_i}{g_i}+w'.
	\end{align}
	 For convenience, let $I=\{1,2,\dots,t\}$ and assume that $\{\log(1-g_1),\dots,\log(1-g_n)\}$ forms a transcendence base for the differential field $F(\{r_i|i\in I\}).$ Then it is well-known that $F(\{r_i|i\in I\})=F(\log(1-g_1),\dots,\log(1-g_n)).$ Since $r'_i=c_i(1-g_i)'/(1-g_i),$ by the Kolchin-Ostrowski Theorem, we have 
	\begin{align}\label{depr_i}
	r_i&=c_i\log(1-g_i)+e_i \qquad\quad~~~ \text{for}\quad i=1,\dots,n\notag\\
	\text{and}\quad r_i&=\sum_{l=1}^nc_{il}\log(1-g_l)+s_i \qquad \text{for}\quad i=n+1,\dots,t,	
	\end{align}
	where $e_i,c_{il}$ are constants and $s_i\in F.$ 
	
	Find $h_1,\dots,h_m$ in $F$ so that $\{\log h_1,\dots,\log h_m\}$ forms a transcendence base for $E$ over $F(\log(1-g_1),\dots,\log(1-g_n)).$ Then $E=F(\log(1-g_1),\dots,\log(1-g_n),\log h_1,\dots,\log h_m).$  Now $w\in E$ and $w'$ is a linear polynomial in $\log(1-g_i),$ therefore, using Proposition 2.2(c) of \cite{ykvrs-2019}, we shall write
	\begin{align}\label{w}
	w=&\sum_{l,p=1}^na_{lp}\log(1-g_l)\log(1-g_p)+\sum_{l=1}^ny_l\log(1-g_l)+\tilde{w},
	\end{align}
	where each $a_{lp}$ is a constant and elements $y_l,\tilde{w}$ belongs to $F(\log h_1,\dots,\log h_m).$ Substituting  Equations \ref{depr_i} and \ref{w} in Equation \ref{eqndel} and equating the coefficients of $\log(1-g_l)$ to $0$, we obtain
	\begin{align}\label{y_l}
	c_l\frac{g'_l}{g_l}+\sum_{i=n+1}^tc_{il}\frac{g'_i}{g_i}+\sum_{p=1}^n(a_{pl}+a_{lp})\frac{(1-g_p)'}{(1-g_p)}+y'_l=0 \in F.
	\end{align}
	In particular, we have $y'_l\in F.$ Since $y_l\in F(\log h_1,\dots,\log h_m),$   from the Kolchin-Ostrowski Theorem, $y_l=\sum_{q=1}^me_{lq}\log h_q+z_l$ for constants $e_{lq}$ and elements $z_l\in F.$ 
	Rewriting Equation \ref{eqndel}, we get
	 	\begin{align}\label{tilde}
	 S=\sum_{i=1}^ne_i\frac{g'_i}{g_i}+\sum_{i=n+1}^ts_i\frac{g'_i}{g_i}+\sum_{l=1}^ny_l\frac{(1-g_l)'}{1-g_l}+\tilde{w}'.
	 \end{align}
		Substituting for $y_l$, we obtain 	
	\begin{align}\label{tildew}
	S=\sum_{i=1}^ne_i\frac{g'_i}{g_i}+\sum_{i=n+1}^ts_i\frac{g'_i}{g_i}+\sum_{l=1}^n\bigg(\sum_{q=1}^me_{lq}\log h_q+z_l\bigg)\frac{(1-g_l)'}{1-g_l}+\tilde{w}'.
	\end{align}
	Thus $\tilde{w}'$ is a linear polynomial in $\log h_q$, for each $q$, over $F.$  Using Proposition 2.2(c) of \cite{ykvrs-2019}, for constants $d_{jp}$ and $w_q,w_0\in F$,  we  write
	\begin{align*}
	\tilde{w}=\sum_{j,q=1}^md_{jq}\log h_j\log h_q+\sum_{q=1}^mw_q\log h_q+w_0.
	\end{align*}

	Substitute $\tilde{w}$ in Equation \ref{tildew} and compare the coefficients of $\log h_q$ to obtain
	\begin{align*}
	\sum_{l=1}^ne_{lq}\frac{(1-g_l)'}{1-g_l}+\sum_{j=1}^m(d_{jq}+d_{qj})\frac{h'_j}{h_j}+w'_q=0.
	\end{align*}
	That is, $
	\sum_{l=1}^ne_{lq}\log(1-g_l)+\sum_{k=1}^m(d_{jq}+d_{qj})\log h_q+w_q$ is a constant in $F.$ Since $\log(1-g_1),\dots,\log(1-g_n),\log h_1,\dots,\log h_m$ forms a transcendence base for $E$ over $F,$ the coefficients $e_{lq},d_{jq}+d_{qj}$ must be $0.$ Therefore, $y_l=z_l\in F,$ $w_q\in C_F$  and $\tilde{w}=\sum_{q=1}^mw_q\log h_q+w_0.$

	Now  Equation \ref{tildew} becomes
	\begin{equation}\label{desireddelexp}
	S=\sum_{i=1 }^ne_i\frac{g'_i}{g_i}+\sum_{i=n+1}^ts_i\frac{g'_i}{g_i}+\sum_{l=1}^ny_l\frac{(1-g_l)'}{1-g_l}+\sum_{q=1}^mw_q\frac{h'_q}{h_q}+w'_0.
	\end{equation}
	From Equations \ref{newr_i} and \ref{depr_i}, for $i=n+1,\dots,t,$ 
	\begin{equation*}
	s'_i=c_i\frac{(1-g_i)'}{1-g_i}-\sum_{l=1}^nc_{il}\frac{(1-g_l)'}{1-g_l} \end{equation*}
and 	from Equation  \ref{y_l}, for $l=1,\dots,n,$  \begin{equation*}
	y'_l=-c_l\frac{g'_l}{g_l}-\sum_{i=n+1}^tc_{il}\frac{g'_i}{g_i}-\sum_{p=1}^n(a_{lp}+a_{pl})\frac{(1-g_p)'}{1-g_p}.
	\end{equation*}
Thus, for some constants $\tilde{a_i}$ and $\tilde{b_l},$ 	\begin{align*}
	s_i&=c_i\log(1-g_i)-\sum_{l=1}^nc_{il}\log(1-g_l)+\tilde{a_i}\quad\text{for}\quad i=n+1,\dots,t\ \  \text{and}\\y_l&=-c_l\log g_l-\sum_{i=n+1}^tc_{il}\log g_i-\sum_{p=1}^n(a_{lp}+a_{pl})\log(1-g_p)+\tilde{b_l}\qquad\text{for}\quad l=1,\dots,n.
	\end{align*}
	
	Now Equation \ref{desireddelexp} becomes
	\begin{align}\label{desireddelexp1}
		S=&\sum_{i=1 }^ne_i\frac{g'_i}{g_i}+\sum_{i=n+1}^t\Big(c_i\log(1-g_i)-\sum_{l=1}^nc_{il}\log(1-g_l)+\tilde{a_i}\Big)\frac{g'_i}{g_i}\notag\\&+\sum_{l=1}^n\Big(-c_l\log g_l-\sum_{i=n+1}^tc_{il}\log g_i-\sum_{p=1}^n(a_{lp}+a_{pl})\log(1-g_p)+\tilde{b_l}\Big)\frac{(1-g_l)'}{1-g_l}+\sum_{q=1}^mw_q\frac{h'_q}{h_q}+w'_0.
\end{align}
	Note that the coefficient of $(g'_i/g_i)\log(1-g_l)$ is same as that of $((1-g_l)'/(1-g_l))\log g_i.$ Therefore, assuming 	\begin{multicols}{2}
		$\lambda_i:=\begin{cases}y_i, \ \mbox{if}\quad 1\leq i\leq n\\ s_i, \ \mbox{if}\quad n+1\leq i\leq t\\ e_{i-t}, \ \mbox{if}\quad t+1\leq i\leq t+n\\
		w_{i-t-n}\ \mbox{if}\quad t+n+1\leq i\leq t+m+n\end{cases}$
		
		\columnbreak
		
		$\mu_i:=\begin{cases}1-g_i, \ \mbox{if}\quad 1\leq i\leq n\\ g_i, \ \mbox{if}\quad n+1\leq i\leq t\\ g_{i-t}, \ \mbox{if}\quad t+1\leq i\leq t+n\\
		h_{i-t-n} \ \mbox{if}\quad t+n+1\leq i\leq t+m+n\end{cases}$
		
	\end{multicols}	
	and $L=\{1,2,\dots,t+n+m\},$ we rewrite Equation \ref{desireddelexp} as  
	\begin{align}\label{desireddelexp2}
S	&=\sum_{i\in L}\lambda_i\frac{\mu'_i}{\mu_i}+w'_0\notag \\
\text{i.e.,}\qquad v	&=\sum_{i\in L}\lambda_i\frac{\mu'_i}{\mu_i}+\sum_{j\in J}a_j\frac{u'_j}{\log u_j}+\sum_{k\in K}b_kv'_ke^{-v_k^2}+w'_0\notag\\
	\text{and}\qquad\lambda'_i&=\tilde{c}_i\frac{(1-\mu_i)'}{1-\mu_i}+\sum_{l\in L}\tilde{c}_{il}\frac{\mu_l'}{\mu_l},
	\end{align}
	for suitably chosen constants $\tilde{c}_i$ and $\tilde{c}_{il}$ with  $\tilde{c}_{il}=\tilde{c}_{li}.$
\end{proof}

%%%%%%%%%%%%%%%%%%%%%%%%%%%%%%%%%%%%
%%%%%%%%%%%%%%%%%%%%%%%%%%%%%%%%%%%%
\bprop\label{suffmainthrm}
Let $F$ be a differential field and $v\in F$. Suppose $v$ admits a $\del-$expression over $F$ of the form\begin{align}
v=\sum_{i\in I}r_i\frac{g'_i}{g_i}+ \sum_{j\in J}a_j\frac{u'_j}{\log(u_j)}+\sum_{k\in K}b_kv'_ke^{-v^2_k}+w',
\end{align}
\begin{equation}
r'_i=c_i\frac{(1-g_i)'}{1-g_i}+\sum_{l\in I}c_{il}\frac{g'_l}{g_l},
\end{equation}
where $c_i$ and $c_{il}$ are constants with $c_{il}=c_{li}$. Then there exists a $\del-$extension $E$ of $F$ containing an antiderivative of $v.$ 
\eprop

\bpf We claim that $E=F(\{\log (1-g_i),\log g_i,\ell_2(g_i),li(u_j),erf(v_k)\})$ contains an antiderivative of $v$.  Since $r_i=c_i\log(1-g_i)+\sum_{l\in L}c_{il}\log g_l+e_i$ for constants $e_i$ and $c_{li}=c_{il}$, it follows that   \begin{align*}
v=\sum_{i\in I}c_i\log(1-g_i)\frac{g'_i}{g_i}+\frac{1}{2}\sum_{i,l\in I}c_{il}\left(\log g_i\log g_l\right)'+\sum_{i\in I}e_i\frac{g'_i}{g_i}+\sum_{j\in J}a_j\frac{u'_j}{\log u_j}+\sum_{k\in K}b_kv'_ke^{-v^2_k}+w'.
\end{align*}
Observe that the element
\begin{align*}
u:=-\sum_{i\in I}c_i\ell_2(g_i)+\frac{1}{2}\sum_{i,l\in I}c_{il}\log g_i\log g_l+\sum_{i\in I}e_i\log g_i+\sum_{j\in J}a_j \ li(u_j)+\sum_{k\in K}b_k \ erf(v_k)+w
\end{align*}
is  an antiderivative of $v$ and that $u\in E$.\epf
%%%%%%%%%%%%%%%%%%%%%%%%%%%%%%%%%%%%%%%%%%%%%%%
%%%%%%%%%%%%%%%%%%%%%%%%%%%%%%%%%%%%%%%%%%%%%%
%%%%%%%%%%%%%%%%%%%%%%%%%%%%%%%%%%%%%%%%%%%%%%%
	 \bt \label{maintheorem}
Let  $E=F(\theta_1,\dots,\theta_n)$ be a transcendental $\del-$extension of $F.$ Then there is an element $u\in E$ with $u'\in F$ if and only if $u'$ admits a $\del-$expression over $F$ of the form \begin{align*}
u'=\sum_{i\in I}r_i\frac{g'_i}{g_i}+ \sum_{j\in J}a_j\frac{u'_j}{\log u_j}+\sum_{k\in K}b_kv'_ke^{-v^2_k}+w',
\end{align*}
\begin{equation*}
r'_i=c_i\frac{(1-g_i)'}{1-g_i}+\sum_{l\in I}c_{il}\frac{g'_l}{g_l},
\end{equation*}
where $a_j,b_k,c_i,$ and $c_{il}$ are constants such that $c_{il}=c_{li}$.\et

\bpf From Theorem \ref{submaindilog}, we know $u'$ satisfies a special $\del-$expression over a logarithmic extension of $F.$ Now apply Propositions \ref{specialtodel} and \ref{suffmainthrm} to complete the proof.\epf

\bpf[\bf Proof of Theorem \ref{NSdilog}.] Apply Theorem \ref{submaindilog}, Theorem \ref{specialtodel} and Proposition \ref{suffmainthrm} with $J$ and $K$ being empty sets.\epf
%%%%%%%%%%%%%%%%%%%%%%%%%%%%%%%%%%%%%%%%%%%%%%%
%%%%%%%%%%%%%%%%%%%%%%%%%%%%%%%%%%%%%%%%%%%%%%%
%%%%%%%%%%%%%%%%%%%%%%%%%%%%%%%%%%%%%%%%%%%%%%%
%%%%%%%%%%%%%%%%%%%%%%%%%%%%%%%%%%%%%%%%%%%%%%%

%%%%%%%%%%%%%%%%%%%%%%%%%%%%%%%%%%%%%%%%%%%%%%%%%%%%%%%%%%%%%%%%%%%%%%%%%%%%%%%%%%%%%%
%%%%%%%%%%%%%%%%%%%%%%%%%%%%%%%%%%%%%%%%%%%%%%%%%%%%%%%%%%%%
\section{Applications}  \label{examples}
\bex
	Consider the differential field $F=\mathbb{C}(z,\log(1+z),\log(z(1-z)(1-z-z^2)))$ and the element 
	\begin{align}
	v=-\frac{(1-z-z^2)'}{1-z-z^2}\log(1+z)+\frac{z'}{z}\log(z(1-z)(1-z-z^2))+v'_0\in F. 
	\end{align}
	Through a lengthy calculation, it was proved in \cite{ykvrs-2019}, pp.231-232 that $v$ admits an antiderivative in a transcendental dilogarithmic-elementary extension of $F$. However, if we let  
	 $g_1:=1-z-z^2,~g_2:=z,~r_1=-\log(1+z)$ and $r_2=\log((z(1-z)(1-z-z^2)))$ then we see that
	\begin{align}
	r'_1&=-\left(\frac{z'}{z}+\frac{(1+z)'}{1+z}\right)+\frac{z'}{z}=-\frac{(1-g_1)'}{1-g_1}+\frac{g'_2}{g_2}\quad\text{and}\notag\\
	r'_2&=\frac{(1-z)'}{1-z}+\frac{(1-z-z^2)'}{1-z-z^2}+\frac{z'}{z}=\frac{(1-g_2)'}{1-g_2}+\frac{g'_1}{g_1}+\frac{g'_2}{g_2},
	\end{align}
	which is in accordance with the Theorem \ref{NSdilog}. Thus it has become immediate that $v$ admits an antiderivative in some transcendental dilogarithmic- elementary extension $E$ of $F.$ \eex
%%%%%%%%%%%%%%%%%%%%%%%%%%%%%%%%%%%%%%%%%%%%
%%%%%%%%%%%%%%%%%%%%%%%%%%%%%%%%%%%%%%%%%%%%
\bex\label{posexp} Let $F=\C(x,e^x)$ be the ordinary differential field with derivation $':=d/dx$ and consider the differential field $E=\C(x,e^x, \log(1-e^x), \ell_2(e^x)).$ We shall now find all elements of $F$ having an antiderivative in $E.$ 

Let $u\in E$ and that $u'\in F.$ Then by the Kolchin -Ostrowski Theorem, for some $w\in \C(x,e^x,\log(1-e^x))$ and constant $c,$ we have $u=c\ell_2(e^x)+w.$ That is,
	\begin{equation} \label{ue^z}
	u'=-c\frac{(e^x)'}{e^x}\log(1-e^x)+w'=-c\log(1-e^x)+w',	
	\end{equation} which is a $\di-$expression over $\C(x,e^x,\log(1-e^x)).$
	Then using Proposition 2.2 (c) of \cite{ykvrs-2019}, we can write $w=c_1\log^2(1-e^x)+w_1\log(1-e^x)+w_0,$ for some constant $c_1$ and elements $w_1,w_0\in F.$ Substituting $w'$ in Equation \ref{ue^z} and comparing the coefficients of $\log(1-e^x),$ we  obtain
	\[w'_1=c-2c_1\frac{(1-e^x)'}{1-e^x}.\] That is, $(w_1-cx)'=2c_1(1-e^x)'/(1-e^x)$ One can show that there is no element $z\in F$ such that $z'=(1-e^x)'/(1-e^x)$ and therefore $c_1=0$.  Hence from Equation \ref{ue^z}, we obtain  that \[u'=w_1\frac{(1-e^x)'}{1-e^x}+w'_0,\quad\text{where }\ w'_1=c=c\frac{(e^x)'}{e^x}.\] Conversely, it is easy to see that if \begin{align*}
	v:&=r\frac{(1-e^x)'}{1-e^x}+w'\in F\\
	&=-r'\log(1-e^x)+(w+r\log(1-e^x))'
	\end{align*} where $r,w\in \C(x,e^x)$ and $r'=c$ for some $c\in \C$ then $\int v=c\ell_2(e^x)+w+r\log(1-e^x)+d\in E$ for some constant $d\in \C.$
	\eex

%%%%%%%%%%%%%%%%%%%%%%%%%%%%%%%%%%%%%%%%%%%%
%%%%%%%%%%%%%%%%%%%%%%%%%%%%%%%%%%%%%%%%%%%%
%%%%%%%%%%%%%%%%%%%%%%%%%%%%%%%%%%%%%%%%%%%%
%%%%%%%%%%%%%%%%%%%%%%%%%%%%%%%%%%%%%%%%%%%%

 \begin{theorem}\label{logarithmicintegrals} If $H\in \C(Y)$ is a non-constant rational function such that $H(\log(x))$ has no antiderivatives in $\C(x,\log(x))$ then  it has no antiderivatives in any transcendental dilogarithmic-elementary extension of $\C(x,\log(x)).$  \end{theorem}

\begin{proof}	
	Suppose on the contrary  that  $H(\log(x))$ has an antiderivative in a transcendental dilogarithmic-elementary extension of $\C(x)(\log(x)).$ Then from Theorem \ref{NSdilog}, we have
	\begin{align}\label{egmainexp}
	H(\log x)=\sum_{i\in I}r_i\frac{g'_i}{g_i}+w',
	\end{align}
	\begin{equation}\label{egexpr_i}
r'_i=c_i\frac{(1-g_i)'}{1-g_i}+\sum_{j\in I}c_{ij}\frac{g'_j}{g_j},
	\end{equation}
	where $r_i, g_i, w\in \C(x,\log(x))$ and $c_i$ and $c_{ij}$ are constants such that  $c_{ij}=c_{ji}$. 
Observe that $(x\log(x)-x)'=\log(x)$ and define $R_0=x$ and for $n\geq 1,$ $R_n:=x\log^n(x)-nR_{n-1}.$ Then  it can be easily verified that  $R'_n=(x\log^n(x)-R_{n-1})'=\log^n(x)$  for $n\geq 1$.  Therefore for any polynomial $P\in \C[Y],$ there is an element $q\in \C(x)[\log(x)]$ with $q'=P(\log(x)).$	Thus, if necessary, we shall suitably replace $w$  and assume that the partial fraction expansion of $H(\log(x))$ over $\C$ is of the form  \begin{align}\label{egexpofH}	H(\log(x))=\sum_{p=1}^l\sum_{q=1}^{m_p}\frac{f_{pq}}{(\log(x)-\a_p)^{q}}.
	\end{align} 
where $f_{pq},\a_p\in \C.$	 
	 
	 Let 
	\begin{align}\label{parfracw}
		w=\sum_{p=1}^s\sum_{q=1}^{n_p}\frac{w_{pq}}{(\log(x)-\b_p)^{q}}+P(\log(x)),
	\end{align}	
	where $P$ is a polynomial over $\overline{\C(x)}$, be the partial fraction expansion of $w$ over $\overline{\C(x)}.$ Note that $P(\log(x))'$ is again a polynomial in $\log(x)$
 over $\overline{\C(x)}.$ Assume for the moment that we have proved $\sum_{i\in I}r_i\frac{g'_i}{g_i}$ is a polynomial in $\log(x)$ over $\overline{\C(x)}.$ Since $$\left(\frac{w_{pq}}{(\log(x)-\b_p)^{q}}\right)'=\frac{w'_{pq}}{(\log(x)-\b_p)^{q}}-\frac{qw_{pq}((1/x)-\b'_p)}{(\log(x)-\b_p)^{q+1}},$$ it then follows from Equation \ref{egmainexp} that \begin{align*}P(\log(x))'&=-\sum_{i\in I}r_i\frac{g'_i}{g_i}\ \ \text{and that }\\ H(\log x)&= (w-P(\log(x)))'.\end{align*}
 Now $P(\log (x))\in \overline{\C(x)}[\log x]$ and $P(\log (x))'\in \C(x)[\log x].$ By Proposition 2.1 (a) of \cite{ykvrs-2019}, there is an element $Q(\log (x))\in \C(x)[\log x]$ such that $Q(\log(x))'=P(\log(x))'.$ Thus,
   $$H(\log x)= (w-Q(\log(x)))'$$ which is a contraction 
and this completes the proof of the theorem.

Now we shall in fact prove that $\sum_{i\in I}r_i\frac{g'_i}{g_i}=\eta \log(x)+\zeta$ for some $\eta, \zeta\in \overline{\C(x)}.$ Let  $g_i=\eta_i\prod_{p=1}^{n}(\log(x)-\b_p)^{a_{ip}}$ and $1-g_i=\xi_i\prod_{p=1}^{n}(\log(x)-\b_p)^{b_{ip}},$ where $\eta_i,\xi_i\in \mathbb{C}(x),$  $\b_p\in \overline{\mathbb{C}(x)}$ and $a_{ip}, b_{ip}$ are integers.  Then 
	
\begin{equation}\label{logderivgi}
\frac{g'_i}{g_i}=\frac{\eta'_i}{\eta_i}+\sum^n_{p=1}a_{ip}\frac{(\log(x)-\b_p)'}{\log(x)-\b_p}\qquad\text{and}\quad \frac{(1-g_i)'}{1-g_i}=\frac{\xi'_i}{\xi_i}+\sum^n_{p=1}b_{ip}\frac{(\log(x)-\b_p)'}{\log(x)-\b_p}
\end{equation}	and thus Equation \ref{egexpr_i} becomes
	\begin{equation}\label{expressionforr'_i}r'_{i}=c_i\frac{\xi_i'}{\xi_i}+\sum_{j\in I}c_{ij}\frac{\eta'_j}{\eta_j}+\sum_{p=1}^n\left(c_ib_{ip}+\sum_{j\in I}c_{ij}a_{ip}\right)\frac{(\log(x)-\b_p)'}{\log(x)-\b_p}.\end{equation}	

Let  $z\in \C(x)(\log(x))$  have a pole of order $m\geq 1$ at $\b\in \overline{\C(x)}$.  Using the partial fraction expansion of $z,$ we find a unique element $\tilde{z}\in \C(x)(\log(x))$ such that $\tilde{z}$ has no pole at $\beta$ and that\begin{equation}\label{shortpar-frac} z=\frac{f_m}{(\log(x)-\beta)^m}+ \cdots+\frac{f_0}{\log(x)-\beta}+\tilde{z},\quad{where }\ f_m\neq 0. \end{equation}
Since $\tilde{z}$ has no pole at $\beta,$ its derivative $\tilde{z}'$ cannot have a pole at $\beta$ either. Since there is no element in $\C(x)$ whose derivative is $1/x$,  we have $\beta'\neq 1/x$ and it follows that \begin{equation}\label{shortpar-fracderiv}z'=\frac{mf_m(\log(x)-\beta)'}{(\log(x)-\beta)^{m+1}}+ \ \text{terms involving lower powers of }\ \frac{1}{\log(x)-\beta}+\tilde{z}'\end{equation} has a pole at $\b$ of order $m+1\geq 2.$ Thus, for any $z\in \C(x)(\log(x)),$ $z'$ has no simple poles. 

Taking $z=r_i,$ we obtain from Equation \ref{expressionforr'_i} that both $r_i$ and $r'_i$ cannot have  poles. Thus for each $i\in I$, $c_ib_{ip}+\sum_{j\in I}c_{ij}a_{ip}=0$ and  \begin{equation} 
	r'_i=c_i\frac{\xi_i'}{\xi_i}+\sum_{j\in I}c_{ij}\frac{\eta'_j}{\eta_j}\in \C(x). \label{egexpr_ij2}
	\end{equation}

Note that $r_i\in \C(x)(\log x)$ and $r'_i\in \C(x).$	 Therefore, by  Kolchin-Ostrowski Theorem,  $r_i=e_i\log(x)+d_i$ for some constants $d_i$ and $e_i.$ Thus 
\begin{align}\sum_{i\in I}r_i\frac{g'_i}{g_i}&=\sum_{i\in I}(e_i\log(x)+d_i)\left(\frac{\eta'_i}{\eta_i}+\sum^n_{p=1}a_{ip}\frac{(\log(x)-\b_p)'}{\log(x)-\b_p}\right)\notag\\  &=\sum_{i\in I}(e_i\log(x)+d_i)\frac{\eta'_i}{\eta_i}+\sum_{i\in I}\sum^n_{p=1} \left(e_i \b_p +d_i\right)a_{ip}\frac{(\log(x)-\b_p)'}{\log(x)-\b_p}+\sum_{i\in I}\sum^n_{p=1}e_ia_{ip}(\log(x)-\b_p)'.\notag\\ &= \eta \log(x)+ \zeta+\sum_{i\in I}\sum^n_{p=1} \left(e_i \b_p +d_i\right)a_{ip}\frac{(\log(x)-\b_p)'}{\log(x)-\b_p},\label{Hlogxrexp}
	\end{align} 
	where $\eta=\sum_{i\in I}e_i\frac{\eta'_i}{\eta_i}\in \overline{\C(x)}$ and $\zeta=\sum_{i\in I} d_i\frac{\eta'_i}{\eta_i}+\sum_{i\in I}\sum^n_{p=1}e_ia_{ip}\left(\frac{1}{x}-\b'_p\right)\in \overline{\C(x)}$. We claim that for each $p,$ $\sum_{i\in I}(e_i \b_p +d_i)a_{ip}=0$ and this would prove that $\sum_{i\in I}r_i\frac{g'_i}{g_i}$ is a polynomial of degree at most one, as desired. 

Suppose that $\beta\in \overline{\C(x)}$ be a  pole of $w$ of order $m\geq 1.$ Then as noted earlier, $w'$ has a pole at $\beta$ of order $m+1\geq 2.$ From Equation \ref{Hlogxrexp}, we observe that $\sum_{i\in I}r_i\frac{g'_i}{g_i}$ can have only simple poles. Now since $H(\log(x)=\sum_{i\in I}r_i\frac{g'_i}{g_i}+w'$ and that the poles of $H(\log(x))\in \C(\log(x))$ are constants,  we obtain that $\beta$ is must also be a constant. That is $\b\in \C.$ Now the poles of $H(\log(x))$ and $w'$ are constants and therefore  the  poles  of $\sum_{i\in I}r_i\frac{g'_i}{g_i}=H(\log(x))-w'$ are constants as well. 

If $\beta_p\in \C$ is a pole of $\sum_{i\in I}r_i\frac{g'_i}{g_i}$ then \begin{equation*}c:=Res_{\beta_p}(H(\log(x)))=Res_{\beta_p}(\sum_{i\in I}r_i\frac{g'_i}{g_i})+ Res_{\beta_p}(w')=\sum_{i\in I}a_{ip}(e_i\beta_p+d_i)(\log(x)-\beta_p)'+w'_{p1},\end{equation*}
where $c\in \C$ and $w_{p1}\in \C(x)$ is the residue of $w$ at $\beta_p.$ Since $\b_p$ has to be a simple pole of $\sum_{i\in I}r_i\frac{g'_i}{g_i},$ we have $d:=\sum_{i\in I}a_{ip}(e_i\beta_p+d_i)$ to be a nonzero constant. Thus we obtain $$c-\frac{d}{x}=w'_{p1} \quad \text{for }\  w_{p1}\in \overline{\C(x)},$$ which contradicts the fact that $1/x$ has no antiderivative in $\overline{\C(x)}$. Therefore $\sum_{i\in I}r_i\frac{g'_i}{g_i}$ has no poles, that is, $\sum_{i\in I}(e_i \b_p +d_i)a_{ip}=0$ for each $p.$ \end{proof}

\br It can be  shown (either through a hand computation or applying Risch algorithm) that $\C(x)(\log(x))$ does not contain any antiderivative of $1/\log(x)$. Thus, taking $H(Y)=1/Y$, we shall apply Theorem \ref{logarithmicintegrals} and prove that the logarithmic integral $\int 1/\log(x) dx$ does not belong  to any transcendental dilogarithmic-elementary extension of $\C(x,\log(x)).$  
\er

%%%%%%%%%%%%%%%%%%%%%%%%%%%%%%%%%%%%%%%%%%%%%%%%%%%%%%%%%%%%%%%%%%%%%%%%%%%%%%%%%%%%%%%%%%%%%%%%%%%%%%%%%%%%%%%%%%%%%%%%%%%%%%%%%%
%%%%%%%%%%%%%%%%%%%%%%%%%%%%%%%%%%%%%%%%%%%%

%%%%%%%%%%%%%%%%%%%%%%%%%%%%%%%%%%%%%%%%%%%%%%%
\section{Logarithmic and Dilogarithmic Identities}\label{sectionidndilog}
%%%%%%%%%%%%%%%%%%%%%%%%%%%%%%%%%%%%%%%%%%%%%%%
%%%%%%%%%%%%%%%%%%%%%%%%%%%%%%%%%%%%%%%%%%%%%%%
%%%%%%%%%%%%%%%%%%%%%%%%%%%%%%%%%%%%%%%%%%%%%%%
%%%%%%%%%%%%%%%%%%%%%%%%%%%%%%%%%%%%%%%%%%%%%%%
%%%%%%%%%%%%%%%%%%%%%%%%%%%%%%%%%%%%%%%%%%%%%%%
%%%%%%%%%%%%%%%%%%%%%%%%%%%%%%%%%%%%%%%%%%%%%%%

%One of the identities that will be used frequently is noted as a remark below. The approach to the remark is similar to the one used by Baddoura in \cite{Badd1}, pp. 924-925.

%%%%%%%%%%%%%%%%%%%%%%%%%%%%%%%%%%%%%%%%%
%\br\label{logetaxi}
%If $f\in F$ such that $x'\neq f$ for all $x\in F$ then for any $\theta$ transcendental over $F,$ the differential field $F(\theta)$ defined by $\theta'=f$ has $C_E=C(F)$ (reference).  

The goal of this section is to establish certain logarithmic and dilogarithmic identities. Let $F\subsetneqq F(\t)$ be differential fields, where $\t$ is transcendental over $F$, $f\in F(\t)$ be a non-zero element and  $\oF$ be an algebraic closure of $F$.  Choose
monic coprime polynomials $P,Q\in F[\t]$ and elements $\eta,\xi\in F$ so that $$f=\eta\frac{P}{Q}\qquad\text{and}\quad 1-f=\frac{Q-\eta P}{Q}.$$ Let $R$ be a monic polynomial such that $\xi R=Q-\eta P$ and observe that $R$ is coprime to both $P$ and $Q$. Let $\a_1,\a_2,\dots,\a_t \in \oF$ be distinct elements such that $$P=\prod^m_{j=1}(\t-\a_j)^{a_j},\quad \frac{1}{Q}=\prod^n_{j=m+1}(\t-\a_j)^{a_j}\quad \text{and}\quad R=\prod^t_{j=n+1}(\t-\a_j)^{b_j},$$
where $a_1,\cdots, a_m$ are positive integers, $a_{m+1},\cdots, a_n$ are negative integers and $b_{n+1},\cdots, b_t$ are positive integers. Then $$f=\eta\prod^t_{i=1}(\t-\a_j)^{a_j}\qquad \text{and}\quad 1-f=\xi\prod^t_{j=1}(\t-\a_j)^{b_j},$$ where $a_{n+1}=\dots=a_t=0$,  $b_1=\dots=b_m=0$ and $b_j=a_j$ for $j=m+1,\dots,n.$  
Let $E$ be any differential field extension of $\overline{F}$, with $C_E=C_F$, containing $\ell_2(\eta), \log \eta, \log \xi$ and $\log(\a_i-\a_j)$ for all $i\neq j$.

\begin{proposition} \label{somelogarithmicidentitites}For any $v_1,\dots, v_t\in E$, the following identities hold: \begin{enumerate}[(i)]\item $\sum_{\substack{j,k=1\\k\neq j}}^t(a_kb_j-a_jb_k)\frac{\alpha_j'-\alpha_k'}{\alpha_j-\alpha_k}v_k =\sum_{k=1}^t\left(b_k\frac{\eta'}{\eta}-a_k\frac{\xi'}{\xi}\right)v_k.$  \label{logetaxi1}\\  
\item $\sum_{\substack{j,k=1\\k\neq j}}^t(a_kb_j-a_jb_k)\log(\a_j-\a_k)v_k =\sum_{k=1}^t\left(b_k\log \eta-a_k\log\xi+c_k\right)v_k,$ where each $c_k$ is a constant.  \label{logetaxi} 
 \end{enumerate}
\end{proposition}

\begin{proof} Consider the expression $$T=\sum_{\substack{j,k=1\\k\neq j}}^t(a_kb_j-a_jb_k)\frac{\alpha_j'-\alpha_k'}{\alpha_j-\alpha_k}v_k.$$ Since $b_1=\dots=b_m=0,$ $a_{n+1}=\dots=a_t=0$ and $a_j=b_j$ for $j=m+1,\dots,n,$ we observe that  $T=T_1+T_2+T_3,$ where
\begin{align*}
&T_1=\sum_{k=1}^ma_k\left(\sum_{j=m+1}^tb_j\frac{\alpha_j'-\alpha_k'}{\alpha_j-\alpha_k}\right)v_k\\&T_2=-\sum_{k=n+1}^tb_k\left(\sum_{j=1}^na_j\frac{\alpha_j'-\alpha_k'}{\alpha_j-\alpha_k}\right)v_k\\&T_3=\sum_{k=m+1}^n\left(a_k\sum_{j=m+1}^tb_j\frac{\alpha_j'-\alpha_k'}{\alpha_j-\alpha_k}-b_k\sum_{j=1}^na_j\frac{\alpha_j'-\alpha_k'}{\alpha_j-\alpha_k}\right)v_k.
\end{align*}
Since $\eta P+\xi R=Q$, if $P(\alpha_k)=0$ for some $k=1,\ldots,m$ then $\xi R(\alpha_k)=Q(\alpha_k)$ and therefore $\dfrac{R(\alpha_k)'}{R(\alpha_k)}-\dfrac{Q(\alpha_k)'}{Q(\alpha_k)}=-\dfrac{\xi'}{\xi}.$  Note that $$\frac{R(\alpha_k)'}{R(\alpha_k)}=\sum^t_{j=n+1}b_j\frac{\alpha'_k-\alpha'_j}{\alpha_k-\alpha_j}\qquad\text{and}\quad \frac{Q(\alpha_k)'}{Q(\alpha_k)}=-\sum^n_{j=m+1}b_j\frac{\alpha'_k-\alpha'_j}{\alpha_k-\alpha_j}.$$
Thus,
\begin{align}\label{part1}\sum_{j=m+1}^tb_j\frac{\alpha_j'-\alpha_k'}{\alpha_j-\alpha_k}=\frac{R(\alpha_k)'}{R(\alpha_k)}-\frac{Q(\alpha_k)'}{Q(\alpha_k)}=-\frac{\xi'}{\xi}.\end{align}
This implies
\begin{align}\label{part1'}T_1=-\sum_{k=1}^ma_k\frac{\xi'}{\xi}v_k.\end{align}
Similarly one shows for $k=n+1,\ldots,t$ that  
\begin{align}\label{part2}\sum_{j=1}^na_j\frac{\alpha_j'-\alpha_k'}{\alpha_j-\alpha_k}=\frac{P(\alpha_k)'}{P(\alpha_k)}-\frac{Q(\alpha_k)'}{Q(\alpha_k)}=-\frac{\eta'}{\eta}\end{align}
and
\begin{align}\label{part2'}
T_2=\sum_{k=n+1}^tb_k\frac{\eta'}{\eta}v_k.
	\end{align}
Since for $k=m+1,\dots,n~ $ we have $a_k=b_k$, it follows that $$a_k\sum_{j=m+1}^nb_j\frac{\a'_j-\a'_k}{\a_j-\a_k}-b_k\sum_{j=m+1}^na_j\frac{\a'_j-\a'_k}{\a_j-\a_k}=0.$$ Since $Q(\alpha_k)=0$ for $k=m+1,\dots,n~ $  we have $\eta P(\alpha_k)=-\xi R(\alpha_k)$. Thus \begin{align}\label{part3}\sum_{j=m+1}^tb_j\frac{\alpha_j'-\alpha_k'}{\alpha_j-\alpha_k}-\sum_{j=1}^na_j\frac{\alpha_j'-\alpha_k'}{\alpha_j-\alpha_k}&=\sum_{j=n+1}^tb_j\frac{\alpha_j'-\alpha_k'}{\alpha_j-\alpha_k}-\sum_{j=1}^ma_j\frac{\alpha_j'-\alpha_k'}{\alpha_j-\alpha_k}\notag\\&=-\frac{P(\alpha_k)'}{P(\alpha_k)}+\frac{R(\alpha_k)'}{R(\alpha_k)}\notag\\&=\frac{\eta'}{\eta}-\frac{\xi'}{\xi}.\end{align}This implies
\begin{align}\label{part3'}T_3=\sum_{k=m+1}^n\left(b_k\frac{\eta'}{\eta}-a_k\frac{\xi'}{\xi}\right)v_k.\end{align}
Add Equations \ref{part1'}, \ref{part2'} and \ref{part3'} to obtain the identity (\ref{logetaxi1}). Take the antiderivative of Equations \ref{part1}, \ref{part2} and \ref{part3}, multiply by $v_k$ and then sum over all $j$ and $k$ with $j\neq k$  to obtain  identity (\ref{logetaxi}).\end{proof}

%To prove the last of the identities, we observe from identity (\ref{logetaxi}) that
% $$S=\sum_{k=1}^t\left(b_k\log \eta-a_k\log\xi\right)\a'_k=\sum_{\substack{j,k=1\\k\neq j}}^t(a_kb_j-a_jb_k)\log(\a_j-\a_k)\a'_k-\sum_{k=1}^tc_k\a'_k,$$ where $c_k$ is a constant.
% Let $w:=\sum_{\substack{j,k=1\\j\neq k}}^t(a_kb_j-a_jb_k)\log(\a_j-\a_k)\a_k-\sum_{k=1}^tc_k\a_k$ be an element in a logarithmic extension of $\oF.$ Then
%\begin{align}\label{T_3}
%	S= -\sum_{\substack{j,k=1\\j\neq k}}^t(a_kb_j-a_jb_k)\frac{\a'_j-\a'_k}{\a_j-\a_k}\a_k+w'.
%\end{align}
%Now, by adding and subtracting $\sum_{\substack{j,k=1\\j\neq k}}^t(a_kb_j-a_jb_k)\frac{\a'_j-\a'_k}{\a_j-\a_k}\a_j$ to the right hand side of above equation, we obtain
%\begin{align}
%	S=-\sum_{\substack{j,k=1\\j\neq k}}^t(a_kb_j-a_jb_k)\frac{\a'_j-\a'_k}{\a_j-\a_k}\a_j+\sum_{\substack{j,k=1\\j\neq k}}^t(a_kb_j-a_jb_k)(\a'_j-\a'_k)+w'.
%\end{align}
%Interchange the index $j$ and $k$ in the first term of above equation and add it to Equation \ref{T_3} and obtain that
%\begin{align}
%S=\frac{1}{2}\bigg(\sum_{\substack{j,k=1\\j\neq k}}^t(a_kb_j-a_jb_k)(\a_j-\a_k)+w\bigg)'.
%\end{align}

%%%%%%%%%%%%%%%%%%%%%%%%%%%%%%%%%%%%%%%
%%%%%%%%%%%%%%%%%%%%%%%%%%%%%%%%%%%%%%%
%%%%%%%%%%%%%%%%%%%%%%%%%%%%%%%%%%%%%%%
%%%%%%%%%%%%%%%%%%%%%%%%%%%%%%%%%%%%%%%

%%%%%%%%%%%%%%%%%%%%%%%%%%%%%%%%%%%%%%%%
%%%%%%%%%%%%%%%%%%%%%%%%%%%%%%%%%%%%%%%%
%%%%%%%%%%%%%%%%%%%%%%%%%%%%%%%%%%%%%%%
%%%%%%%%%%%%%%%%%%%%%%%%%%%%%%%%%%%%%%%%
%%%%%%%%%%%%%%%%%%%%%%%%%%%%%%%%%%%%%%%%
%%%%%%%%%%%%%%%%%%%%%%%%%%%%%%%%%%%%%%%%

\begin{proposition}\label{dilogarithmicidentity} If deg$(P)\leq$ deg$(Q)$ then 
\begin{enumerate}[(i)]
	\item $S_1:=\sum_{k=1}^t(a_k\ell_2(\eta)+a_k\log\eta\log\xi-\frac{1}{2}b_k\log^2\eta)$ is a constant in $F.$\label{easythings}
	\item $S_2:=\sum_{k=1}^t(a_k\log\xi-b_k\log\eta)$ is a constant in $F.$\label{easythings1}
	 \item \label{newdilogidentity} For some constants\footnote{In \cite{Badd1}, p.923, Baddoura established a similar identity for dilogarithmic integrals in terms of their Bloch-Wigner-Spence function and our proof of identity (\ref{l_2fprop}) uses similar techniques.}  $c,d_k$ and $e$, \begin{align*}\ell_2(f(\theta))={}& \ell_2(\eta)-\sum_{\substack{j,k=1\\k\neq j}}^t a_jb_k\ell_2\left( \frac{\theta-\alpha_j}{\theta-\alpha_k}\right) - \frac{1}{2}\sum_{\substack{j,k=1}}^ta_jb_k\log^2(\theta-\alpha_k)\\&-\sum_{k=1}^{t}b_k\log(\theta-\alpha_k)\log\eta-e\log\eta+\sum_{k=1}^td_k\log(\t-\a_k)+c.\end{align*}\label{l_2fprop}
\item For any $j,k$, there are constants $c_{jk}$ and $d_{jk}$ such that\begin{align*}
	\ell_2\left(\frac{\t-\a_j}{\t-\a_k}\right)=&-\ell_2\left(\frac{\t-\a_k}{\t-\a_j}\right)+\log(\t-\a_j)\log(\t-\a_k)\\
	&-\frac{1}{2}\left(\log^2(\t-\a_j)+\log^2(\t-\a_k)\right)+d_{jk}\log\left(\frac{\t-\a_j}{\t-\a_k}\right)+c_{jk}.
	\end{align*} \label{inverserelation}
\end{enumerate} 

\end{proposition}

\begin{proof}
Note that  if deg$(P)<$ deg$(Q)$ then $\sum_{k=1}^tb_k=0$ and $\log\xi$ is a constant. Also $\ell_2(\eta)=-\log\xi\log\eta+c$ for a constant $c.$ Therefore, $S_1=\sum_{k=1}^ta_kc$ and $S_2=\sum_{k=1}^ta_k\log\xi$ are both constants. On the other hand, if deg$(P)=$ deg$(Q)$ then $\sum_{k=1}^ta_k=0$ and either $\sum_{k=1}^tb_k=0$ or $\eta=1.$ Thus either $S_1=S_2=0$ or $S_1=-(1/2)\sum_{k=1}^tb_k\log^2\eta$ and $S_2=-\sum_{k=1}^tb_k\log\eta$ are both constants. This proves (\ref{easythings})
 and (\ref{easythings1}).

%%%%%%%%%%%%%%%%%%%%%%%%%%%%%%%%%%%%%%%%%%%%%%%%%
%%%%%%%%%%%%%%%%%%%%%%%%%%%%%%%%%%%%%%%%%%%%%%%%%
%%%%%%%%%%%%%%%%%%%%%%%%%%%%%%%%%%%%%%%%%%%%%%%%%
Since 
	\begin{equation*}
	\ell'_2(f(\theta))=-\frac{f(\t)'}{f(\t)}\log (1-f(\t)),
	\end{equation*}
	we shall replace $f(\t)$ and $1-f(\t)$ with their partial fraction expansions, then $\log(1-f(\t))=\log\xi+\sum_{k=1}^tb_k\log(\t-\a_k)+e$ for some constant $e,$ and rearrange the terms to obtain
	\begin{align} \label{l_2f}
	\ell_2'(f(\t))=&-\frac{\eta'}{\eta}\log\xi-\sum_{k=1}^t(a_k\log\xi-b_k\log\eta)\frac{\t'-\a'_k}{\t-\a_k}-\sum_{j,k=1}^ta_jb_k\frac{\theta'-\alpha_j'}{\theta-\alpha_j}\log(\t-\alpha_k)\notag\\&-\bigg( \sum_{k=1}^tb_k\log(\t-\alpha_k)\log\eta\bigg)'-e\bigg(\frac{\eta'}{\eta}+\sum_{k=1}^ta_k\frac{\t'-\a'_k}{\t-\a_k}\bigg).
	\end{align}
	From the definition of dilogarithmic integral, observe
	\begin{align}\label{l_2frac}
	\sum_{\substack{j,k=1\\ k\neq j}}^ta_jb_k\ell_2'\left( \frac{\t-\alpha_j}{\t-\alpha_k}\right)&=-\sum_{\substack{j,k=1\\ k\neq j}}a_jb_k\bigg(\frac{\t'-\a'_j}{\t-\a_j}-\frac{\t'-\a'_k}{\t-\a_k}\bigg)\log\Big(\frac{\a_j-\a_k}{\t-\a_k}\Big)\notag\\&=-\sum_{\substack{j,k=1\\ k\neq j}}(a_kb_j-a_jb_k)\log(\a_j-\a_k)\frac{\t'-\a'_k}{\t-\a_k}-\sum_{\substack{j,k=1\\ k\neq j}}a_kb_jc_{jk}\frac{\t'-\a'_k}{\t-\a_k}\notag\\&~~+\sum_{\substack{j,k=1\\ k\neq j}}a_jb_k\bigg(\frac{\t'-\a'_j}{\t-\a_j}-\frac{\t'-\a'_k}{\t-\a_k}\bigg)(\log(\t-\a_k)+c_{jkk}),
		\end{align}
	where constants $c_{jk}=\log(\a_k-\a_j)-\log(\a_j-\a_k)$ and $c_{jkk}=\log(\frac{\a_j-\a_k}{\t-\a_k})-\log(\a_j-\a_k)+\log(\t-\a_k).$
	Using Proposition \ref{somelogarithmicidentitites} (\ref{logetaxi}),  Equation \ref{l_2frac} can be written as
	\begin{align*}
	\sum_{\substack{j,k=1\\ k\neq j}}^ta_jb_k\ell_2'\left( \frac{\t-\alpha_j}{\t-\alpha_k}\right)&=-\sum_{k=1}^t(b_k\log\eta-a_k\log\xi+c_k)\frac{\t'-\a'_k}{\t-\a_k}-\sum_{\substack{j,k=1\\ k\neq j}}a_kb_jc_{jk}\frac{\t'-\a'_k}{\t-\a_k}\notag\\&~~+\sum_{\substack{j,k=1\\ k\neq j}}a_jb_k\bigg(\frac{\t'-\a'_j}{\t-\a_j}-\frac{\t'-\a'_k}{\t-\a_k}\bigg)(\log(\t-\a_k)+c_{jkk}),
	\end{align*}
	where $c_k$ are constants. For constant $e_k:=-c_k+\sum_{j=1,j\neq k}^t(-a_kb_jc_{jk}-a_jb_kc_{jkk}+a_kb_jc_{kjj}),$ we have
	\begin{align*}
	\sum_{\substack{j,k=1\\ k\neq j}}^ta_jb_k\ell_2'\left( \frac{\t-\alpha_j}{\t-\alpha_k}\right)=&\sum_{k=1}^t(a_k\log\xi-b_k\log\eta)\frac{\t'-\a'_k}{\t-\a_k}+\sum_{\substack{j,k=1\\ k\neq j}}a_jb_k\bigg(\frac{\t'-\a'_j}{\t-\a_j}-\frac{\t'-\a'_k}{\t-\a_k}\bigg)\log(\t-\a_k)\notag\\&
	+\sum_{k=1}^te_k\frac{\t'-\a'_k}{\t-\a_k}.
	\end{align*}
	Using the above equation, we shall rewrite Equation \ref{l_2f} as
	\begin{align}\label{dilogidentity}\ell_2'(f(\theta))=&-\frac{\eta'}{\eta}\log\xi-\sum_{\substack{j,k=1\\k\neq j}}^t a_jb_k\ell_2'\left( \frac{\theta-\alpha_j}{\theta-\alpha_k}\right) -\sum_{\substack{j,k=1}}^ta_jb_k\frac{\t-\alpha_k'}{\t-\alpha_k}\log(\theta-\alpha_k)\notag\\&-\left(\sum_{k=1}^{t}b_k\log(\theta-\alpha_k)\log\eta\right)'-e\frac{\eta'}{\eta}+\sum_{k=1}^td_k\frac{\t'-\a'_k}{\t-\a_k},\end{align}
	where $d_k=e_k-ea_k.$ Observe that if deg$(P)<$ deg$(Q)$ then $\log \xi$ is a constant and $(\eta'/\eta)\log \xi=(\log\xi\log\eta)',$ if deg$(P)=$ deg$(Q)$ and $\eta=1$ then  $\frac{\eta'}{\eta}\log\xi=0$ and finally, if $\eta\neq1$ then $\xi=1-\eta$ and $\frac{\eta'}{\eta}\log\xi=-\ell_2'(\eta).$ Thus in any event, by taking the antiderivative,  Equation \ref{dilogidentity} yields the identity (\ref{l_2fprop}).

%%%%%%%%%%%%%%%%%%%%%%%%%%%%%%%%%%%%%%%%%%%%%
%%%%%%%%%%%%%%%%%%%%%%%%%%%%%%%%%%%%%%%%%%%%%
%We note one more basic identity for dilogarithmic integrals in the following remark.
%%%%%%%%%%%%%%%%%%%%%%%%%%%%%%%%%%%%%%%%%%%%

%%%%%%%%%%%%%%%%%%%%%%%%%%%%%%%%%%%%%%%%%%%%%
Note  that
	\begin{align*}\ell_2'\left(\frac{\t-\a_j}{\t-\a_k}\right)&=-\left(\frac{\t'-\a_j'}{\t-\a_j}-\frac{\t'-\a_k'}{\t-\a_k}\right)\log\left(\frac{\a_j-\a_k}{\t-\a_k}\right)\\&
	=-\left(\frac{\t'-\a_j'}{\t-\a_j}-\frac{\t'-\a_k'}{\t-\a_k}\right)(\log(\a_j-\a_k)-\log(\t-\a_k)+\delta_{k})\end{align*}
	and
	\begin{align*}
	\ell_2'\left(\frac{\t-\a_k}{\t-\a_j}\right)&=-\left(\frac{\t'-\a_k'}{\t-\a_k}-\frac{\t'-\a_j'}{\t-\a_j}\right)\log\left(\frac{\a_k-\a_j}{\t-\a_j}\right)\\
	&=-\left(\frac{\t'-\a_k'}{\t-\a_k}-\frac{\t'-\a_j'}{\t-\a_j}\right)(\log(\a_j-\a_k)-\log(\t-\a_j)+\delta_j),
	\end{align*}
	where $\delta_j$ and $\delta_k$ are constants. Adding the above two equation gives
	\begin{align*}
	\ell_2'\left(\frac{\t-\a_j}{\t-\a_k}\right)+\ell_2'\left(\frac{\t-\a_k}{\t-\a_j}\right)=&\left(\log(\t-\a_j)\log(\t-\a_k)\right)'-\frac{\t'-\a_j'}{\t-\a_j}\log(\t-\a_j)\\&-\frac{\t'-\a_k'}{\t-\a_k}\log(\t-\a_k)+d_{jk}\left(\frac{\t'-\a_j'}{\t-\a_j}-\frac{\t'-\a_k'}{\t-\a_k}\right),
	\end{align*}
	where $d_{jk}=\delta_j-\delta_k.$
	Rearrange the above terms and take the antiderivative to prove the identity (\ref{inverserelation}).\end{proof}
%%%%%%%%%%%%%%%%%%%%%%%%%%%%%%%%%%%%%%%%%%%%%%%
%%%%%%%%%%%%%%%%%%%%%%%%%%%%%%%%%%%%%%%%%%%%%%%
%%%%%%%%%%%%%%%%%%%%%%%%%%%%%%%%%%%%%%%%%%%%%%%%%%%%%%%%%%%%%%%%%%%%%%%%%%%%%%%%%%%%%%

\subsection{Algebraic independence of certain dilogarithmic integrals.}
From the identity (\ref{inverserelation}), it is clear that $\ell_2(\frac{\t-\a_j}{\t-\a_k})$ and $\ell_2(\frac{\t-\a_k}{\t-\a_j})$ are algebraically dependent over any differential field containing $\log(\t-\a_j)$ and $\log(\t-\a_k)$.  However, in the next lemma, we shall show that the set $\{\ell_2\left(\frac{\t-\a_j}{\t-\a_k}\right);k>j,\a_j\neq\a_k \ \ \text{for}\ \ j\neq k \}$ is algebraically independent over logarithmic extensions of $F(\t).$

\bl
Let $F(\t)\supset F$ be differential fields, $\t$ be transcendental over $F$, $C_{F(\t)}=C_F$ and assume that either $\t'\in F$ or $\t'/\t\in F.$ Let $\a_1,\dots,\a_t,~t\geq3$ be distinct elements in $F.$ Then the set $\{\ell_2\left(\frac{\t-\a_j}{\t-\a_k}\right);k>j\}$ is algebraically independent over the logarithmic extension  $E=F(\t)(\{\log(\a_j-\a_k),\log(\t-\a_j);j,k=1,\dots,t\}).$  
\el
\bpf
Suppose the set $\{\ell_2\left(\frac{\t-\a_j}{\t-\a_k}\right);k>j\}$ is algebraically dependent over $E.$ Since for each $k>j,$ the derivative $\ell_2'\left(\frac{\t-\a_j}{\t-\a_k}\right)\in E,$ we shall apply Kolchin-Ostrowski Theorem and obtain 
\[\ell_2\left(\frac{\t-\a_1}{\t-\a_2}\right)=\sum_{\substack{j,k=1\\k>j,k\neq 2}}^tc_{jk}\ell_2\left(\frac{\t-\a_j}{\t-\a_k}\right)+v,\]
where each $c_{jk}$ is a constant and $v\in E.$
Taking the derivatives, we obtain 
\begin{align}\label{algindeqn}
-\left(\frac{\t'-\a'_1}{\t-\a_1}-\frac{\t'-\a'_2}{\t-\a_2}\right)\log\left(\frac{\a_1-\a_2}{\t-\a_2}\right)=-\sum_{\substack{j,k=1\\k>j,k\neq 2}}^tc_{jk}\left(\frac{\t'-\a'_j}{\t-\a_j}-\frac{\t'-\a'_k}{\t-\a_k}\right)\log\left(\frac{\a_j-\a_k}{\t-\a_k}\right)+v'.
\end{align}
Let $F^*$ be a subfield of $E$ such that $\t$ is transcendental over $F^*$ and $F^*(\t)=F(\t)(\{\log(\a_j-\a_k)\}.$ Since $\t'$ or $\t'/\t\in F,$ the elements $\log(\t-\a_1),\dots,\log(\t-\a_n)$ are algebraically independent over $F^*(\t),$ except  when $\t'/\t\in F$ and $\a_1=0.$ 
Since $\log(\t-\a_2)$ is transcendental over the field $F_2:=F^*(\t)(\{\log(\t-\a_j);j\neq 2\})$, $v$ must be a polynomial in $\log(\t-\a_2)$ of degree at most $2$ and thus we shall write $v=c_1\log^2(\t-\a_2)+v_1\log(\t-\a_2)+v_0,$ where $c_1\in C_F$ and $v_1,v_0\in F_2.$ Compare the coefficients of $\log(\t-\a_2)$ in the above equation and obtain 
\[\frac{\t'-\a'_1}{\t-\a_1}-\frac{\t'-\a'_2}{\t-\a_2}=2c_1\frac{\t'-\a'_2}{\t-\a_2}+v'_1.\]
It is obvious that $c_1=-1/2$ and for a constant $c,$ $v_1=\log(\t-\a_1)+c.$ Thus the Equation \ref{algindeqn} becomes
\begin{align}\label{afteralpha2}&-\left(\frac{\t'-\a'_1}{\t-\a_1}-\frac{\t'-\a'_2}{\t-\a_2}\right)\log\left(\a_1-\a_2\right)=\notag\\&-\sum_{\substack{j,k=1\\k>j,k\neq 2}}^tc_{jk}\left(\frac{\t'-\a'_j}{\t-\a_j}-\frac{\t'-\a'_k}{\t-\a_k}\right)\log\left(\frac{\a_j-\a_k}{\t-\a_k}\right)+\frac{\t'-\a'_2}{\t-\a_2}(\log(\t-\a_1)+c)+v'_0.\end{align}

Since $\log(\t-\a_1)$ is transcendental over $F_1=F^*(\t)(\{\log(\t-\a_j);j\neq 1,2\}),$ except when $\t'/\t\in F$ and $\a_1=0,$  we observe that  if $\a_1\neq 0$ then $v_0$ must be a polynomial in $\log(\t-\a_1)$ of degree at most $2$. Thus we shall write $v_0=c_2\log^2(\t-\a_1)+w_1\log(\t-\a_1)+w_0,$ for $c_2\in C_F$ and $w_1,w_0\in F_1.$  Comparing the coefficients of $\log(\t-\a_1)$ in the above equation, we shall obtain 
\[\frac{\t'-\a'_2}{\t-\a_2}+2c_2\frac{\t'-\a'_1}{\t-\a_1}+w'_1=0.\]
But, irrespective of whether $\t'$ or $\t'/\t\in F$, $w_1$ has no poles. Thus we have arrived at a contradiction. Hence, we shall assume $\a_1=0.$

Let $\t'/\t=x',$ where $x\in F$ and $\a_1=0$.  Then the Equation \ref{afteralpha2} becomes
\begin{align}\label{aftertheta}&\bigg(\frac{\t'-\a'_2}{\t-\a_2}-x'\bigg)\log\left(-\a_2\right)=-\sum_{\substack{j,k=1\\k>j,k\neq 2}}^tc_{jk}\bigg(\frac{\t'-\a'_j}{\t-\a_j}-\frac{\t'-\a'_k}{\t-\a_k}\bigg)\log\bigg(\frac{\a_j-\a_k}{\t-\a_k}\bigg)+\frac{\t'-\a'_2}{\t-\a_2}(x+c)+v'_0.\end{align}
Proceeding as earlier, since $\log(\t-\a_3)$ is transcendental over the differential field $F_3:=F^*(\t)(\{\log(\t-\a_j);j\neq 1,2,3\}),$ we conclude that $v_0$ is a polynomial in $\log(\t-\a_3)$ of the form $v_0=c_3\log^2(\t-\a_3)+s_1\log(\t-\a_3)+s_0,$ for $c_3\in C_F$ and $s_1,s_0\in F_3.$ Substituting $v_0$ and comparing the coefficients of $\log(\t-\a_3) $ in Equation \ref{aftertheta}, we obtain
\begin{align}
	c_{13}\left(x'-\frac{\t'-\a'_3}{\t-\a_3}\right)+c_{23}\left(\frac{\t'-\a'_2}{\t-\a_2}-\frac{\t'-\a'_3}{\t-\a_3}\right)+2c_3\frac{\t'-\a'_3}{\t-\a_3}+s'_1=0.
\end{align}
Comparing the poles, we get $c_{13}=2c_3,$ $c_{23}=0$ and $s_1=-c_{13}x+d,$ where $d\in C_F.$ Thus, the Equation \ref{aftertheta} reduces to
\begin{align}\bigg(\frac{\t'-\a'_2}{\t-\a_2}-x'\bigg)\log\left(-\a_2\right)=&-\sum_{\substack{j,k=1\\k>j,k\neq 2,3}}^tc_{jk}\left(\frac{\t'-\a'_j}{\t-\a_j}-\frac{\t'-\a'_k}{\t-\a_k}\right)\log\left(\frac{\a_j-\a_k}{\t-\a_k}\right)+\frac{\t'-\a'_2}{\t-\a_2}(x+c)\notag\\&+c_{13}\left(x'-\frac{\t'-\a'_3}{\t-\a_3}\right)\log(-\a_3)+(c_{13}x+d)\frac{\t'-\a'_3}{\t-\a_3}+s_0.\end{align}Repeating this procedure and comparing  the coefficients of $\log(\t-\a_k)$ for $ k>3,$ we obtain $c_{jk}=0$ for all $2<j<k\leq t$ and
\begin{align}\label{eqnthet3}\bigg(\frac{\t'-\a'_2}{\t-\a_2}-x'\bigg)\log\left(-\a_2\right)=&-\sum_{k>2}^tc_{1k}\left(x'-\frac{\t'-\a'_k}{\t-\a_k}\right)\log(-\a_k)+\frac{\t'-\a'_2}{\t-\a_2}(x+c)\notag\\&+\sum_{k>2}^t(c_{1k}x+d_{k})\frac{\t'-\a'_k}{\t-\a_k}+t_0,\end{align}
where $d_{k}\in C_F$ and element $t_0\in F.$ Now since $(\t'-\a'_2)/(\t-\a_2)=x'+(x'\a_2-\a'_2)/(\t-\a_2),$ we shall compare the poles of $\t-\a_2$ in Equation \ref{eqnthet3} and obtain that $\log(-\a_2)=x+c$.  That is, $(\a_2)'/(\a_2)=x'=(\t')/(\t)$ and $(\t/\a_2)'=0.$ This contradicts our assumption that  $C_{F(\t)}=C_F.$\epf

%%%%%%%%%%%%%%%%%%%%%%%%%%%%%%%%%%%%%%%%%%%%%%%%%%%%%%%%%%%%%%%%%%%%%%%%%%%%%%%%%%%%%%
%%%%%%%%%%%%%%%%%%%%%%%%%%%%%%%%%%%%%%%%%%%%%%%%%%%%%%%%%%%%%%%%%%%%%%%%%%%%%%%%%%%%%%

%%%%%%%%%%%%%%%%%%%%%%%%%%%%%%%%%%%%%%%%%%%%%%%
%%%%%%%%%%%%%%%%%%%%%%%%%%%%%%%%%%%%%%%%%%%%%%%
%%%%%%%%%%%%%%%%%%%%%%%%%%%%%%%%%%%%%%%%%%%%%%%
%%%%%%%%%%%%%%%%%%%%%%%%%%%%%%%%%%%%%%%%%%%%%%%
%%%%%%%%%%%%%%%%%%%%%%%%%%%%%%%%%%%%%%%%%%%%%%%%%%%%%%%%%%%%%%%%%%%%%%%%%%%%%%%%%%%%%%

%%%%%%%%%%%%%%%%%%%%%%%%%%%%%%%%%%%%%%%%%%%%%%%%%%%%%%%%%%%%
%%%%%%%%%%%%%%%%%%%%%%%%%%%%%%%%%%%%%%%%%%%%%%%%%%%%%%%%%%%%
%%%%%%%%%%%%%%%%%%%%%%%%%%%%%%%%%%%%%%%%%%%%%%%%%%%%%%%%%%%%

%\bibliographystyle{amsart-num}
%\bibliography{intfinite}
	\bibliography{<your-bib-database>}

\end{document}